\documentclass[hidelinks,11pt]{amsart}

\usepackage{amsmath,amssymb,amsfonts,hyperref,graphicx,tikz,mathrsfs,amsaddr,blindtext}

\usepackage[labelsep=period, font=footnotesize, labelfont=bf]{caption}

\makeatletter
\@namedef{subjclassname@2020}{2020 Mathematics Subject Classification}
\makeatother

\usepackage[margin=1.1in,footskip=0.25in]{geometry}

\newtheorem{theorem}{Theorem}[section]
\newtheorem{lemma}[theorem]{Lemma}

\theoremstyle{definition}

\newtheorem{corollary}[theorem]{Corollary}
\newtheorem{conjecture}[theorem]{Conjecture}

\newtheorem{proposition}[theorem]{Proposition}
\newtheorem{remark}[theorem]{Remark}

\numberwithin{equation}{section}

\date{\today}

\begin{document}


\title[On integral variations for roots of the Laplacian matching polynomial of graphs]
{On integral variations for roots of the Laplacian matching polynomial of graphs}


\author{Yi Wang $^1$, Hai-Jian Cui $^1$, Sebastian M. Cioaba $^2$}
\address{1. School of Mathematical Sciences, Anhui University, Hefei 230601, Anhui, China \\2. Department of Mathematical Sciences, University of Delaware, Newark, DE 19716-2553, USA}
\email{wangy@ahu.edu.cn, cuihj@stu.ahu.edu.cn, cioaba@udel.edu}
\thanks{{\it Funding.} Supported by the National Natural Science Foundation of China (No. 12171002, 12331012)}
\date{\today}
\begin{abstract}
In this paper, we study the Laplacian matching polynomial of a graph and the effect of adding edges to a graph on the roots (called Laplacian matching roots) of this polynomial. In particular, we investigate the conditions under which the Laplacian matching roots change by integer values. We prove that the Laplacian matching root integral variation in one place is impossible and the Laplacian matching root integral variation in two places is also impossible under some constraints.
\end{abstract}
\maketitle

\section{Introduction}

There are several polynomials associated with a graph such as the characteristic polynomial, the chromatic polynomial, the matching polynomial, and the Tutte polynomial among others. Investigating such  polynomials and their connections to the graph properties is an important topic in graph theory, see \cite{EMM} for example. In this paper, we study how the roots of the Laplacian matching polynomial of a graph change by integer quantities after adding an edge. Before stating our results, we introduce some notations and definitions. 

Throughout this paper, all graphs are assumed to be finite, undirected, and without loops or multiple edges. 
Let $G$ be a graph with vertex set $V(G)=\left\{v_1,\ldots,v_n\right\}$ and edge set $E(G)=\left\{e_1,\ldots,e_m\right\}$. 
The {\em complement} of $G$ is the graph whose vertex set is $V(G)$ and whose edge set is the complement of $E(G)$ in the set of all unordered pairs of vertices. For a vertex $v$ of $G$, we denote by $N(v)$ the set of all vertices of $G$ adjacent to $v$. {\em The degree} of $v$ is defined as $|N(v)|$, and is denoted by $d(v)$. The maximum degree of the vertices of $G$ is denoted by $\Delta(G)$. When $e$ is an edge of the complement of $G$, $G+e$ denotes the graph obtained by adding $e$ to $G$. For a subset $W$ of $V(G)$, we use $G[W]$ to denote the subgraph of $G$ induced by $W$.
For a subset $M$ of $E(G)$, we use $V(M)$ to denote the set of vertices of $G$ each of which is an endpoint of one of the edges in $M$. If no two distinct edges in $M$ share a common endpoint, then $M$ is called a {\em matching} of $G$. The set of matchings of $G$ is denoted by $\mathcal{M}(G)$. A matching with $j$ edges is called a {\em $j$-matching}. For an integer $j$,  we denote by $\phi_j(G)$ the number of $j$-matchings of $G$, with the convention that $\phi_0(G)=1$. 

The {\it matching polynomial} of $G$ is defined as
\begin{equation}\label{eq:Mdef}
\mathscr{M}_G(x)=\sum_{M\in\mathcal{M}(G)}(-1)^{|M|}x^{|V(G)\setminus V(M)|}=\sum_{j=0}^{\lfloor n/2\rfloor}(-1)^j\phi_j(G)x^{n-2j}.
\end{equation}
Motivated by studies in statistical physics, Heilmann and Lieb introduced this polynomial in \cite{HL} and proved that for a graph $G$ with maximum degree $\Delta=\Delta(G)\geq 2$, the roots of $\mathscr{M}_G(x)$ are real and lie in the interval $(-2\sqrt{\Delta-1}, 2\sqrt{\Delta -1})$. A graph $G$ is a forest if and only if its matching polynomial is identical to the characteristic polynomial of its adjacency matrix, see \cite[Ch.2]{Godsil1}.

A signed graph $(G,s)$ consists of a graph $G$ and a sign function $s:E(G)\rightarrow \{1,-1\}$. The {\em adjacency matrix} of the signed graph $(G,s)$ is the $V(G)\times V(G)$ matrix whose $(u,v)$-entry equals $s(u,v)$, if $uv\in E(G)$, and $0$, otherwise. Signed graphs and their adjacency matrices have been used by Huang \cite{Huang} in his celebrated proof of the Sensitivity Conjecture, by Marcus, Spielman and Srivastava \cite{MSS} in their breakthrough result showing the existence of infinite bipartite Ramanujan graphs of every degree and have close connections to other areas of mathematics such as equiangular lines. Godsil and Gutman \cite{Godsil3} proved that the average of the adjacency characteristic polynomials of all signed graphs with underlying graph $G$ equals the matching polynomial of $G$. Using these results, Marcus, Spielman and Srivastava \cite{MSS} developed the theory of interlacing polynomials in their study of Ramanujan graphs.

Motivated by the above relation between the adjacency characteristic polynomial and the matching polynomial, it is natural to ask: {\em What is the average of Laplacian characteristic polynomials of all signed graphs with underlying graph $G$?} Mohammadian \cite{Moham} studied this question and named this average polynomial {\it the Laplacian matching polynomial} of $G$, denoted by $\mathscr{LM}_G(x)$, and proved that it has the following expression
\begin{equation}\label{LMdef}
\mathscr{LM}_G(x)=\sum_{M\in\mathcal{M}(G)}(-1)^{|M|}\prod_{v\in V(G)\setminus V(M)}(x-d(v)).
\end{equation}

Mohammadian \cite{Moham} proved that the Laplacian matching polynomial and the Laplacian characteristic polynomial of a graph $G$ are identical if and only if $G$ is a forest. Independently, Zhang and Chen \cite{chen} also studied this polynomial and called it {\it the average Laplacian polynomial} of $G$. Equations \eqref{eq:Mdef} and \eqref{LMdef} imply that if $G$ is a $d$-regular graph, then $\mathscr{LM}_G(x)=\mathscr{M}_G(x-d)$. Hence, the Laplacian matching polynomial may be also interpreted as a generalization of the matching polynomial from regular graphs to general graphs. Mohammadian \cite{Moham} also showed that all Laplacian matching roots of $G$ are real and nonnegative. We denote them by $\lambda_1(G)\geq \ldots \geq \lambda_n(G)\geq 0$. Moreover, Wan, Wang, and Mohammadian \cite{WWM} proved the following interlacing result.
\begin{theorem} {\em \cite{WWM}}
If $G$ is a graph of order $n$ and $e$ is an edge of the complement of $G$, then the Laplacian matching roots of $G+e$ interlace those of $G$, that is,
 \begin{equation} \label{interlace}
 \lambda_1(G+e)\geq \lambda_1(G) \geq \lambda_2(G+e)  \geq \lambda_2(G) \geq \cdots \geq \lambda_{n}(G+e) \geq \lambda_n(G).
 \end{equation}
\end{theorem}

Note that a similar interlacing result (with respect to deleting or adding a vertex) holds for the matching polynomial $\mathscr{M}_G(x)$, see \cite[Corollary 6.1]{Godsil1}.

Using \eqref{LMdef}, it is not difficult to observe that the sum of all Laplacian matching roots of $G$ equals $\sum_{v \in V(G)}d(v)$. Thus, for any edge $e$ of the complement of $G$,
\begin{equation} \label{sum}
\sum_{i=1}^n \lambda_i(G+e)-\sum_{i=1}^n \lambda_i(G)=2.
\end{equation}
From \eqref{interlace} and \eqref{sum},  we deduce that by adding an edge to a graph, none of the Laplacian matching roots can decrease, and that the sum of those roots will increase by two. To investigate the integrality of Laplacian matching roots, in this paper, we only consider  the circumstances under which the addition of an edge to a graph will cause the Laplacian matching roots to change only by integer quantities. There are just two possible
cases that can happen as follows:
\begin{itemize}
	\item[(A)] one root of $\mathscr{LM}_G(x)$ increases by 2 and the other $n-1$ roots of $\mathscr{LM}_G(x)$ remain unchanged;
	\item [(B)] two roots of $\mathscr{LM}_G(x)$ increase by 1 and the other $n-2$ roots of $\mathscr{LM}_G(x)$ remain unchanged.
	\end{itemize}
We refer to (A) and (B) by saying the {\it Laplacian matching root integral variation} (abbreviated by LMRIV) occurs to $G$ in one place by  adding an edge and LMRIV occurs to $G$ in two places, respectively. 

In this paper, we study the Laplacian matching root integral variation LMRIV for connected graphs. We prove that LMRIV in one place is impossible and LMRIV in two places is also impossible if $\frac{g(G)}{c(G)}> \frac{7}{6}$, where $g(G)$ is the girth of $G$ and $c(G)$ is  the dimension of cycle space of $G$. We conjecture that LMRIV in two places is impossible for all connected graphs.

Our work is motivated by similar studies done on the eigenvalues of Laplacian matrix. So \cite{W} studied this type of roots variation for the Laplacian characteristic polynomial and characterized the graphs where the Laplacian spectral integral variation occurs in one place. Kirkland \cite{S} determined all graphs with Laplacian spectral integral variation occurring in two places. Another motivation for our work comes from the study of graphs whose matching polynomial has integer roots, see \cite{ACGKGN}.

\section{Preliminaries}

In this section, we collect some concepts and known results about the matching polynomial and the Laplacian matching polynomial for later use. Let $G$ be a graph of order $n$. If $S\subset V(G)$, then $G-S$  is the graph obtained from $G$ by deleting the vertices in $S$ together with all  edges incident to any vertex in $S$. The matching polynomial satisfies the following basic identity, which is called the expansion formula of $\mathscr{M}(G,x)$ at vertex $v$.
\begin{proposition}{\em \cite[Thm. 1.1.1(c)]{Godsil1}} \label{exp}
If $G$ is a graph and $v\in V(G)$, then
\begin{equation}
\mathscr{M}_{G}(x)=x\mathscr{M}_{G-\{v\}}(x)-\sum_{u\in N(v)}\mathscr{M}_{G-\{u,v\}}(x).
\end{equation}
\end{proposition}

The subdivision $S(G)$ of a graph $G$ is the graph obtained from $G$ by replacing every edge $e=\left\{a,b\right\}$ of $G$ with two edges $\left\{a,v_e\right\}$ and $\left\{b,v_e\right\}$ along with the new vertex $v_e$ corresponding to the edge $e$. The following result describes a connection between the matching polynomial and the Laplacian matching polynomial and is a useful tool to deal with the Laplacian matching roots of a graph.
\begin{theorem}{\em \cite{WWM, Yan, chen}} \label{subdivision}
If $G$ is a graph, then 
\begin{equation}
\mathscr{M}_{S(G)}(x)=x^{|E(G)|-|V(G)|}\mathscr{LM}_{G}(x^2).
\end{equation}
\end{theorem}

Another useful result due to Zhang and Chen \cite{chen} gives a combinatorial interpretation for the coefficients of the Laplacian matching polynomial and is an analogue of the Sachs coefficient theorem for the characteristic polynomial of a graph, see \cite[Corollary 2.3.3]{CRS} or \cite{Sachs}. A {\em TU-subgraph} of $G$ is a subgraph whose components are trees or unicyclic graphs.
Suppose that a TU-subgraph $H$ of $G$ consists of $s$ unicyclic graphs and trees $T_1,\ldots,T_t$.
The {\em weight} of $H$ is defined as
$$\omega(H)=2^s\prod_{i=1}^{t}|T_i|,$$
where $|T_i|$ is the order of $T_i$.
\begin{theorem}{\em \cite{chen}}  \label{cominter}
If $G$ is a graph of order $n$ with Laplacian matching polynomial $\mathscr{LM}_{G}(x)=\sum_{i=0}^{n}(-1)^ib_ix^{n-i}$, then
\begin{equation}
b_i=\omega(\mathscr{H}_i)=\sum_{H\in \mathscr{H}_i}\omega(H),
\end{equation}
for $i=1,2,\ldots,n$, where $\mathscr{H}_i$ denotes the set of all the TU-subgraphs of $G$ with $i$ edges.
\end{theorem}

The following three theorems from \cite{WWM, chen} will be used in the later section.  
\begin{theorem}{\em \cite{chen}} \label{zero}
Let $G$ be a connected graph. The graph $G$ is a tree if and only if $\lambda_n(G)=0$.
\end{theorem}

\begin{theorem}{\em \cite{WWM}} \label{larroot1}
If $G$ is a connected graph, then
\begin{equation}
\lambda_1(G)\geq\Delta(G)+1.
\end{equation}
Equality holds if and only if $G$ is a star.
\end{theorem}

\begin{theorem}{\em \cite{WWM}}\label{larroot2}
If $G$ is a connected graph and $e$ is an edge of its complement, then $\lambda_1(G+e)$ has multiplicity one and $\lambda_1(G+e)>\lambda_1(G)$.
\end{theorem}

\section{Main results}

The purpose of this section is to investigate the LMRIV. We prove first that LMRIV in one place is impossible. In what follows, we assume that $G$ is a connected graph with $V(G)=\left\{v_1,\ldots,v_n\right\}$ and $E(G)=\left\{e_1,\ldots,e_m\right\}$,  and always use $R(G)=\left\{\lambda_1,\ldots,\lambda_n\right\}$ to denote  the multiset of the Laplacian matching roots of $G$, where $\lambda_1 \geq \cdots \geq \lambda_n \geq 0.$
\begin{theorem}\label{LMRIV1}
The LMRIV will not occur in one place.
\end{theorem}
\begin{proof} We use proof by contradiction. Assume that LMRIV occurs to $G$ in one place by adding a new edge $e=v_iv_j$. By Theorem \ref{larroot2}, the largest Laplacian matching root must have changed by 2, and therefore,  $R(G+e)=\left\{\lambda_1+2,\lambda_2,\ldots,\lambda_n\right\}$.
Theorem \ref{subdivision} implies that
\begin{equation}
\mathscr{M}_{S(G)}(x)=x^{m-n}\mathscr{LM}_{G}(x^2)=x^{m-n}\prod_{\ell=1}^n(x^2-\lambda_\ell)
\end{equation}
and
\begin{equation}  \label{S(G+e)exp1}
\begin{split}
  \mathscr{M}_{S(G+e)}(x)& =x^{m-n+1}\mathscr{LM}_{G+e}(x^2) \\
  &=x^{m-n+1}(x^2-\lambda_1-2)\prod_{\ell=2}^n(x^2-\lambda_\ell)\\
  &=x\mathscr{M}_{S(G)}(x)-2x^{m-n+1}\prod_{\ell=2}^n(x^2-\lambda_\ell).
\end{split}
\end{equation}
One the other hand, by Proposition \ref{exp}, we have that
\begin{equation} \label{S(G+e)exp2}
  \mathscr{M}_{S(G+e)}(x)=x\mathscr{M}_{S(G)}(x)-\mathscr{M}_{S(G)-v_i}(x)-\mathscr{M}_{S(G)-v_j}(x).
\end{equation}
Combining \eqref{S(G+e)exp1} and \eqref{S(G+e)exp2}, we deduce that
\begin{equation} \label{S(G+e)exp3}
\mathscr{M}_{S(G)-v_i}(x)+\mathscr{M}_{S(G)-v_j}(x)=2x^{m-n+1}\prod_{\ell=2}^n(x^2-\lambda_\ell).
\end{equation}
By comparing the coefficient of $x^{m+n-3}$ on two sides of \eqref{S(G+e)exp3}, we observe that
$$4m-d(v_i)-d(v_j)=2\sum_{\ell=2}^n\lambda_\ell=2(2m-\lambda_1),$$
which implies that  $d(v_i)+d(v_j)=2\lambda_1$. This contradicts Theorem \ref{larroot1}, completing the proof.
\end{proof}

Now, we focus on considering the case in which the LMRIV occurs to $G$ in two places.  By Theorem \ref{larroot2}, the largest matching root must be changed. In what follows, we always denote another changed root by $\lambda_k$.

\begin{theorem}\label{LMRIV2}
If the $LMRIV$ occurs to $G$ in two places by adding a new edge $e=v_iv_j$ and the changed roots of $G$ are $\lambda_1$ and $\lambda_k$, then
$$\lambda_1+\lambda_k=d(v_i)+d(v_j)+1,$$
$$\lambda_1\lambda_k=d(v_i)d(v_j).$$
\end{theorem}
\begin{proof}
Write $R(G+e)=\left\{\lambda_1+1,\lambda_2,\ldots,\lambda_{k-1},\lambda_k+1,\lambda_{k+1},\ldots,\lambda_n\right\}$. By Theorem 2.2, we get that
\begin{equation} \label{S(G)}
\mathscr{M}_{S(G)}(x)=x^{m-n}\mathscr{LM}_{G}(x^2)=x^{m-n}\prod_{\ell=1}^n(x^2-\lambda_\ell)
\end{equation}
and
\begin{equation} \label{twoS(G+e)exp1}
\begin{split}
 \mathscr{M}_{S(G+e)}(x)=& x^{m-n+1}\mathscr{LM}_{G+e}(x^2)\\
 =&x^{m-n+1}(x^2-\lambda_1-1)(x^2-\lambda_k-1)\prod_{\ell\neq 1,k}(x^2-\lambda_\ell)\\
 =&x\mathscr{M}_{S(G)}(x)-x^{m-n+1}(x^2-\lambda_1)\prod_{\ell\neq 1,k}(x^2-\lambda_\ell)\\
 &-x^{m-n+1}(x^2-\lambda_k)\prod_{\ell\neq 1,k}(x^2-\lambda_\ell)+x^{m-n+1}\prod_{\ell\neq 1,k}(x^2-\lambda_\ell).
\end{split}
\end{equation}
In addition, it follows from  Proposition \ref{exp} that
\begin{equation} \label{twoS(G+e)exp2}
\mathscr{M}_{S(G+e)}(x)=x\mathscr{M}_{S(G)}(x)-\mathscr{M}_{S(G)-v_i}(x)-\mathscr{M}_{S(G)-v_j}(x).
\end{equation}
Combining the \eqref{twoS(G+e)exp1} and \eqref{twoS(G+e)exp2}, one can deduce that
\begin{equation} \label{twoS(G+e)exp3}
\begin{split}
&\mathscr{M}_{S(G)-v_i}(x)+\mathscr{M}_{S(G)-v_j}(x)\\
=&x^{m-n+1}(x^2-\lambda_1)\prod_{\ell\neq 1,k}(x^2-\lambda_\ell)+x^{m-n+1}(x^2-\lambda_k)\prod_{\ell\neq 1,k}(x^2-\lambda_\ell)\\
&-x^{m-n+1}\prod_{\ell\neq 1,k}(x^2-\lambda_\ell)\\
=&(2x^2-\lambda_1-\lambda_k-1)x^{m-n+1}\prod_{\ell\neq 1,k}(x^2-\lambda_{\ell}).
\end{split}
\end{equation}
Note that both $S(G)-v_i$ and  $S(G)-v_j$ contain $m+n-1$ vertices, and contain $2m-d(v_i)$ and $2m-d(v_j)$ edges, respectively.
Further, by comparing the coefficients of $x^{m+n-3}$ on two sides of \eqref{twoS(G+e)exp3}, we observe that
$$4m-d(v_i)-d(v_j)=\lambda_1+\lambda_k+1+2\sum_{\ell \neq 1, k} \lambda_{\ell}=4m-\lambda_1-\lambda_k+1,$$
which implies that
$$\lambda_1+\lambda_k=d(v_i)+d(v_j)+1.$$

Now, we are going to prove the second statement. Recall that  $\phi_2(G)$ denote the number of the $2$-matchings in $G$. For any $u\in V(G)$, it is clear that
\begin{equation}\label{2-matching number1}
\phi_2(S(G))-\phi_2(S(G)- u)=(2m-d(u)-1)d(u).
\end{equation}
By comparing the coefficients of $x^{m+n-4}$ on two sides of \eqref{S(G)}, we observe that
 \begin{equation}\label{2-matching number2}
 \phi_2(S(G))=\sum_{1\leq s<t\leq n}\lambda_s\lambda_t.
 \end{equation}
For $v_i,v_j\in V(G)$, combining \eqref{2-matching number1} and \eqref{2-matching number2}, we have
 $$\phi_2(S(G)- v_i)=\sum_{1\leq s<t\leq n}\lambda_s\lambda_t-(2m-d(v_i)-1)d(v_i)$$
 and
 $$\phi_2(S(G)- v_j)=\sum_{1\leq s<t\leq n}\lambda_s\lambda_t-(2m-d(v_j)-1)d(v_j).$$
 Hence, writing $\phi(v_i,v_j)=\phi_2(S(G)-v_i)+\phi_2(S(G)-v_j)$, we have
\begin{equation}\label{2-matching number3}
 \phi(v_i,v_j)=2\sum_{1\leq s<t\leq n}\lambda_s\lambda_t-(2m-d(v_i)-1)d(v_i)-(2m-d(v_j)-1)d(v_j).
\end{equation}
By comparing the coefficients of $x^{m+n-5}$ on two sides of \eqref{twoS(G+e)exp3}, we have
\begin{equation}\label{coefficient1}
  \phi(v_i,v_j)=2\sum_{s,t\neq 1,k}\lambda_s\lambda_t+(\lambda_1+\lambda_k+1)\sum_{\ell \neq 1,k}\lambda_{\ell}.
\end{equation}
Therefore, it follows from  \eqref{2-matching number3} and \eqref{coefficient1}  that
\begin{equation}\label{coefficient2}
(\lambda_1+\lambda_k-1)\sum_{ \ell \neq 1,k}\lambda_{\ell}+2\lambda_1\lambda_k=2m(d(v_i)+d(v_j))-(d^2(v_i)+d^2(v_j)+d(v_i)+d(v_j)).
\end{equation}
Noting that $\lambda_1+\lambda_k=d(v_i)+d(v_j)+1$ and  $\sum \limits_{\ell \neq 1,k}\lambda_{\ell}=2m-(\lambda_1+\lambda_k) $, we can deduce that
$$\lambda_1\lambda_k=d(v_i)d(v_j),$$
as desired.
\end{proof}

As applications of Theorem \ref{LMRIV2}, in the following two consequences, we give some sufficient conditions for the LMRIV not occurring in two places.
\begin{corollary} \label{tree}
For each tree $T$, the LMRIV will not occur in two places.
\end{corollary}
\begin{proof}
If the $LMRIV$ occurs to $T$ in two places by adding a new edge $e=v_iv_j$, by Theorem \ref{zero}  and Theorem \ref{larroot2}, the changed roots are $\lambda_1$ and $\lambda_n (=0)$.  It follows from  Theorem \ref{LMRIV2} that
$$d(v_i)d(v_j)=\lambda_1\lambda_n=0,$$
which implies that $d(v_i)=0$ or $d(v_j)=0$, contradiction.
\end{proof}

\begin{corollary} \label{1+3}
For two nonadjacent vertices $v_i$ and $v_j$, if $d(v_i)+d(v_j)\leq 3$, then the LMRIV will not occur to $G$ in two places by adding a new edge $e=v_iv_j$.
\end{corollary}
\begin{proof}
Assume that $LMRIV$ occurs to $G$ in two places by adding the edge $e=v_iv_j$ with $d(v_i)+d(v_j)\leq3$. It follows from Theorem \ref{LMRIV2} and the Vieta's formulas that
$$\lambda_1=\frac{(d(v_i)+d(v_j)+1)+\sqrt{(d(v_i)+d(v_j)+1)^2-4d(v_i)d(v_j)}}{2}<d(v_i)+d(v_j)+1\leq 4.$$
Theorem \ref{larroot1} states that $\lambda_1\geq \Delta(G)+1$, which implies that $\Delta(G)<3$. Since $G$ contains
two vertices $v_i$ and $v_j$ such that $d(v_i)+d(v_j)\leq 3$, we conclude that $G$ is a path.  By  Corollary \ref{tree}, the LMRIV will not occur in two places, a contradiction.
\end{proof}

Next, we will keep on discussing the LMRIV occurring in two places using another tool. By Theorem \ref{cominter}, we may let
\begin{equation}
\mathscr{LM}_{G}(x)=\sum_{i=0}^n(-1)^ib_ix^{n-i}
\end{equation}
and 
\begin{equation}
\mathscr{LM}_{G+e}(x)=\sum_{i=0}^n(-1)^i\tilde{b}_ix^{n-i}.
\end{equation}
If the LMRIV occurs to $G$ in two places by adding a new edge $e=v_iv_j$ and the changed roots of $G$ are $\lambda_1$ and $\lambda_k$, then
 \begin{equation} \label{quotient}
\frac{\widetilde{b}_n}{b_n}=\frac{\mathscr{LM}_{G+e}(0)}{\mathscr{LM}_{G}(0)}=\frac{(\lambda_1+1)(\lambda_k+1)}{\lambda_1 \lambda_k}.
 \end{equation}
For convenience, denote by  $\mathscr{H}(G)$ (or $\mathscr{H}$) the set of all the  TU-subgraphs of $G$ with $n$ edges, and denote by $\mathscr{T}(G)$ the set of all spanning trees of $G$. By Theorem \ref{cominter},  we have that 
\begin{equation}
b_n=\omega(\mathscr{H}(G))=\sum_{H\in \mathscr{H}(G)}\omega(H).
\end{equation}


 \begin{lemma} \label{2+2}
 For two nonadjacent vertices $v_i$ and $v_j$, if $d(v_i)=d(v_j)=2$, then the LMRIV will not occur to $G$ in two places by adding a new edge $e=v_iv_j$.
 \end{lemma}

\begin{proof}
Assume that $LMRIV$ occurs to $G$ in two places by adding the edge $e=v_iv_j$ with $d(v_i)=d(v_j)=2$. It follows from Theorem \ref{LMRIV2} that $\lambda_1=4$ and $\lambda_k=1$. By Theorem \ref{larroot1}, we can deduce that $\Delta(G) \leq 3$. If $\Delta(G)= 3$, then $\lambda_1 =\Delta(G)+1$, and so Theorem \ref{larroot1} states that $G$ is a star. This contradicts Corollary \ref{tree}.
If $\Delta(G)\leq 2$, then $G$ is a tree unless that $G$ is a cycle $C_n$. By Corollary \ref{tree}, we only consider the case that $G=C_n \  (n \geq 3)$.  By Theorem \ref{cominter}, it is clear that $b_n=2$ and $\widetilde{b}_n=2(n+1)$.
By \eqref{quotient}, $$\frac{2(n+1)}{2}=\frac{\widetilde{b}_n}{b_n}=\frac{(\lambda_1+1)(\lambda_k+1)}{\lambda_1 \lambda_k}=\frac{5}{2}.$$ This contradiction completes the proof.
\end{proof}

Let us introduce more notations and definitions for later use.  We always suppose that $G$ is a connected graph of order $n \geq 3$ and size $m$. The girth of $G$ is denoted by $g(G)$, and the dimension of cycle space of $G$ is denoted by $c(G)$. We know that $c(G)=m-n+1$. A connected graph $G$ is unicyclic (bicyclic, resp.) if $c(G)=1$ ($c(G)=2$,  resp.).
Let $\pi=\{V_1, \ldots, V_{p(\pi)}\}$ be a partition of $V(G)$. Denote by $G_i$ the subgraph of $G$ induced by $V_i$ for $1 \leq i \leq p(\pi)$.
A partition $\pi$ is called {\it TU-admissible} if, for every $i$ with $ 1\leq i \leq p(\pi)$, $G_i$ is a connected graph with $c(G_i) \geq 1$. Denote by $\Omega$ the set of all these TU-admissible partitions. We use $\mathscr{H}_{\pi}(G)$ to denote the subset of $\mathscr{H}(G)$ consisting of all TU-subgraphs of size $n$ whose components are corresponding to $\pi$. 
Note that for each $H \in \mathscr{H}(G)$, all components of $H$ are unicyclic and $H$ is a spanning subgraph of $G$, because that $H$ is a TU-subgraph of size $n$.
Then, $\mathscr{H}(G)$ can be partitioned as $\{\mathscr{H}_{\pi}(G) \}_{\pi \in \Omega}$.
If $p(\pi)=1$, then we simply use $\mathscr{H}_1(G)$ instead of $\mathscr{H}_{\pi}(G)$, which denotes the set consisting of all unicyclic spanning subgraphs of $G$.

\begin{lemma} \label{cou}
If $G$ is a connected graph with $c(G) \geq 1$, then $\frac{|\mathscr{T}(G)|}{|\mathscr{H}_1(G)|} \geq \frac{g(G)}{c(G)}$.
\end{lemma}
\begin{proof}  We count the pairs $(T,U)$ consisting of unicyclic spanning subgraphs $U$ and spanning trees $T$ of $G$ satisfying $E(T) \subset E(U)$.  On the one hand, the number of such pairs is given by $|\mathscr{T}(G)|c(G)$. On the other hand, the  number of pairs  $(T,U)$ is at least $|\mathscr{H}_1(G)|g(G)$.  This completes the proof.
\end{proof}

\begin{remark} \label{rmk}
For a connected subgraph $H$ of $G$ with $c(H) \geq 1$, note that $g(H) \geq g(G), c(H) \leq c(G)$ and so $\frac{g(H)}{c(H)} \geq \frac{g(G)}{c(G)}$.
\end{remark}

\begin{lemma} \label{1+1}
	Let $G$ be a connected graph with $\frac{g(G)}{c(G)}>1$. For two nonadjacent vertices $v_i$ and $v_j$, if $d(v_i)=1$ or $d(v_j)=1$, then the LMRIV will not occur to $G$ in two places by adding a new edge $e=v_iv_j$.
\end{lemma}

\begin{proof}  Assume that $LMRIV$ occurs to $G$ in two places by adding the edge $e=v_iv_j$ with $d(v_i)=1$. Denote by $f$ the edge incident to $v_i$. By Corollary \ref{1+3}, we can assume that $d(v_j) \geq 3$.
	 It follows from \eqref{quotient} and Theorem \ref{LMRIV2} that $$\frac{\widetilde{b}_n}{b_n}=\frac{(\lambda_1+1)(\lambda_k+1)}{\lambda_1\lambda_k}=1+\frac{1}{d(v_i)}+\frac{1}{d(v_j)}+\frac{2}{d(v_i)d(v_j)}\leq 3.$$
	To get a contradiction, we prove that  $\frac{\widetilde{b}_n}{b_n}>3$.
 Applying Theorem \ref{cominter}, we have
 \begin{equation*}
 \frac{\widetilde{b}_n}{b_n} =\frac{\omega(\mathscr{H}(G+e))}{\omega(\mathscr{H}(G))}=\frac{\omega(\mathscr{H}^e(G+e))+\omega(\mathscr{H}^f(G+e))+\omega(\mathscr{H}^{e,f}(G+e))}{\omega(\mathscr{H}(G))},
\end{equation*}
where $\mathscr{H}^e(G+e)$ ($\mathscr{H}^f(G+e)$, $\mathscr{H}^{e,f}(G+e)$, resp.) is the subset of $\mathscr{H}(G+e)$ consisting of the TU-subgraphs of size $n$ containing
$e$ but no $f$ (containing
$f$ but no $e$, containing $e$ and $f$, resp.). Clearly, $\omega(\mathscr{H}^f(G+e))=\omega(\mathscr{H}(G))$. Note that  every TU-subgraph in $\mathscr{H}^e(G+e)$ can be obtained from some TU-subgraph of $G$ by replacing $f$ by $e$, and vice versa.
It is not hard to see that  $\omega(\mathscr{H}^e(G+e))=\omega(\mathscr{H}(G))$.  Therefore, it is enough to show that
\begin{equation} \label{H(G+e)}
\frac{\omega(\mathscr{H}^{e,f}(G+e))}{\omega(\mathscr{H}(G))} > 1.
\end{equation}

Let $\pi=\{V_1, \ldots, V_p\}$ be a TU-admissible partition of $V(G)$.  We say that $\pi$ is of Type I if $v_i$ and $v_j$ lie in the same element of $\pi$. Otherwise,  $\pi$ is called of Type II.
 Recall that $\mathscr{H}_{\pi}(G)$ is the subset of $\mathscr{H}(G)$ consisting of the TU-subgraphs of size $n$ whose components are corresponding to $\pi$, and  $\mathscr{H}(G)$ can be partitioned as $\{\mathscr{H}_{\pi}(G) \}_{\pi \in \Omega}$.  For each $\mathscr{H}_{\pi}(G)$, we define the subset $\sigma(\mathscr{H}_{\pi}(G))$ of $\mathscr{H}^{e,f}(G+e)$ as follows.

If $\pi$ is of Type I, without loss of generality, assume that $v_i, v_j \in V_1$. Denote by $G_i$ ($\widetilde{G}_i$, resp.) the subgraph of $G$ ($G+e$, resp.) induced by $V_i$ for $i =1, \ldots, p$. In this situation, $\sigma(\mathscr{H}_{\pi}(G))$ is defined to be the subset of $\mathscr{H}^{e,f}(G+e)$ consisting of the TU-subgraphs of size $n$ whose components are corresponding to $\pi$.
 Note that for any $\widetilde{H}\in \sigma(\mathscr{H}_{\pi}(G))$, all components of $\widetilde{H}$ are unicyclic. Therefore, the component $\widetilde{H}[V_1]$ of $\widetilde{H}$ corresponding to $V_1$ is comprised of a spanning tree $\widetilde{G}_1-v_i$ together with edge $e$ and $f$, and the component $\widetilde{H}[V_i]$ of $H$ corresponding to $V_i$ is a unicyclic spanning subgraph of $\widetilde{G}_i$ for $i=2, \ldots, p$, which is preserved while we add $e$ to $G$. Therefore, we can deduce that
\begin{equation} \label{qu geq 1}
\begin{split}
\frac{\sum_{\widetilde{H} \in \sigma(\mathscr{H}_{\pi}(G))} \omega(\widetilde{H})}{\sum_{H \in \mathscr{H}_{\pi}(G) } \omega(H)}&=\frac{2^p |\mathscr{T}(\widetilde{G}_1-v_i)| \Pi_{i=2}^p|\mathscr{H}_1(\widetilde{G}_i)|}{2^p\Pi_{i=1}^p|\mathscr{H}_1(G_i)|}\\
&=\frac{|\mathscr{T}(\widetilde{G}_1-v_i)|}{|\mathscr{H}_1(G_1)|}\\
&=\frac{|\mathscr{T}(G_1)|}{|\mathscr{H}_1(G_1)|}>1,
\end{split}
\end{equation}
where the last inequality follows from Lemma \ref{cou} and Remark \ref{rmk}.

If $\pi$ is of Type II, without loss of generality, assume that $v_i \in V_1$ and $v_j \in V_2$. We still use $G_i$ ($\widetilde{G}_i$, resp.) to denote  the subgraph of $G$ ($G+e$, resp.) induced by $V_i$ for $i =1, \ldots, p$.
In this situation,  $\sigma(\mathscr{H}_{\pi}(G))$ is defined to be the subset of $\mathscr{H}^{e,f}(G+e)$ consisting of all TU-subgraphs $\widetilde{H}$ of size $n$ satisfying the following conditions:
 \begin{itemize}
 	\item The components of $\widetilde{H}$  are corresponding to the partition $\{V_1\cup V_2, V_3, \ldots, V_p\}$;
 	\item $\widetilde{H}[V_1 \cup V_2]-e$ exactly have two components which correspond to $V_1$ and $V_2$ respectively. Equivalently, $e$ connects a spanning tree of $\widetilde{G}_1$ and a unicyclic spanning subgraph of $\widetilde{G}_2$, or $e$ connects a spanning tree of $\widetilde{G}_2$ and a unicyclic spanning subgraph of $\widetilde{G}_1$.
 \end{itemize}
Note that $G_i=\widetilde{G}_i$ for $i=1, \ldots, p$ as $\pi$ is of Type II.
Therefore, we  can deduce that
 \begin{equation} \label{qu geq 2}
 \begin{split}
 \frac{\sum_{\widetilde{H} \in \sigma(\mathscr{H}_{\pi}(G))} \omega(\widetilde{H})}{\sum_{H \in \mathscr{H}_{\pi}(G) } \omega(H)}&=\frac{2^{p-1}(|\mathscr{T}(\widetilde{G}_1)||\mathscr{H}_1(\widetilde{G}_2)|+|\mathscr{T}(\widetilde{G}_2)||\mathscr{H}_1(\widetilde{G}_1)|) \Pi_{i=3}^p|\mathscr{H}_1(\widetilde{G}_i)|}{2^p\Pi_{i=1}^p|\mathscr{H}_1(G_i)|}\\
 &=\frac{|\mathscr{T}(G_1)|}{2|\mathscr{H}_1(G_1)|}+\frac{|\mathscr{T}(G_2)|}{2|\mathscr{H}_1(G_2)|} >1,
 \end{split}
 \end{equation}
where the last inequality follows from Lemma \ref{cou} and Remark \ref{rmk}.

We are now ready to prove \eqref{H(G+e)}. By  the above definition of   $\sigma(\mathscr{H}_{\pi}(G))$, it is clear that $\sigma(\mathscr{H}_{\pi}(G)) \cap\sigma(\mathscr{H}_{\pi'}(G)) =\emptyset$ if $\pi \neq \pi'$.
Note that  $\omega(\mathscr{H}(G))= \sum_{\pi \in \Omega}\sum_{H \in \mathscr{H}_{\pi}(G) } \omega(H)$ and $\omega(\mathscr{H}^{e,f}(G+e))= \sum_{\pi \in \Omega}\sum_{\widetilde{H} \in \sigma(\mathscr{H}_{\pi}(G)) }  \omega(\widetilde{H})$. To establish \eqref{H(G+e)}, it suffices to show that $$\frac{\sum_{\widetilde{H} \in \sigma(\mathscr{H}_{\pi}(G))} \omega(\widetilde{H})}{\sum_{H \in \mathscr{H}_{\pi}(G) } \omega(H)}> 1$$  for each  $\pi \in \Omega$, which has been provided by \eqref{qu geq 1} and \eqref{qu geq 2}. This completes the proof.
\end{proof}

We are now ready to present the main theorem in this paper.

\begin{theorem} \label{mainthm}
	Let $G$ be a connected graph with $\frac{g(G)}{c(G)}>\frac{7}{6}$. Then, the LMRIV will not occur to $G$ in two places by adding a new edge.
\end{theorem}

\begin{proof}
Assume that $LMRIV$ occurs to $G$ in two places by adding the edge $e=v_iv_j$.
 By combining Corollary  \ref{1+3}, Corollary \ref{2+2} and Lemma \ref{1+1}, we may assume that $2 \leq d(v_i) < d(v_j)$.  It follows from \eqref{quotient} and Theorem \ref{LMRIV2} that $$\frac{\widetilde{b}_n}{b_n}=\frac{(\lambda_1+1)(\lambda_k+1)}{\lambda_1\lambda_k}=1+\frac{1}{d(v_i)}+\frac{1}{d(v_j)}+\frac{2}{d(v_i)d(v_j)}\leq \frac{13}{6}$$
with equality holds if and only if $d(v_i)=2$ and $d(v_j)=3$.

	To get a contradiction, we  prove that  $\frac{\widetilde{b}_n}{b_n}>\frac{13}{6}$.
Applying Theorem \ref{cominter}, we have
\begin{equation*}
\frac{\widetilde{b}_n}{b_n} =\frac{\omega(\mathscr{H}(G+e))}{\omega(\mathscr{H}(G))}=\frac{\omega(\mathscr{H}^e(G+e))+\omega(\mathscr{H}^{\hat{e}}(G+e))}{\omega(\mathscr{H}(G))},
\end{equation*}
where $\mathscr{H}^e(G+e)$ ($\mathscr{H}^{\hat{e}}(G+e)$, resp.) is the subset of $\mathscr{H}(G+e)$ consisting of all TU-subgraphs of size $n$ containing
$e$ (containing no $e$, resp.). Clearly, $\omega(\mathscr{H}^{\hat{e}}(G+e))=\omega(\mathscr{H}(G))$.  Therefore, it is enough to show that
\begin{equation}  \label{add2}
\frac{\omega(\mathscr{H}^{e}(G+e))}{\omega(\mathscr{H}(G))} > \frac{7}{6}.
\end{equation}

The remaining proof is similar to the proof of Lemma \ref{1+1}. For the sake of completeness, we include the details here.

Let $\pi=\{V_1, \ldots, V_p\}$ be a TU-admissible partition of $V(G)$. We say that $\pi$ is of Type I if $v_i$ and $v_j$ lie in the same element of $\pi$. Otherwise,  $\pi$ is called of Type II. For each $\mathscr{H}_{\pi}(G)$, we define the subset $\sigma(\mathscr{H}_{\pi}(G))$ of $\mathscr{H}^{e}(G+e)$ as follows. If $\pi$ is of Type I, assume that $v_i, v_j \in V_1$. Denote by $G_i$ ($\widetilde{G}_i$, resp.) the subgraph of $G$ ($G+e$, resp.) induced by $V_i$ for $i =1, \ldots, p$. In the situation, $\sigma(\mathscr{H}_{\pi}(G))$ is defined to be the subset of $\mathscr{H}^{e}(G+e)$ consisting of all TU-subgraphs of size $n$ whose components are corresponding to $\pi$. 

Note that for any $\widetilde{H}\in \sigma(\mathscr{H}_{\pi}(G))$, all components of $\widetilde{H}$ are unicyclic. Therefore, the component $\widetilde{H}[V_1]$ of $\widetilde{H}$ corresponding to $V_1$ is comprised of a spanning tree $\widetilde{G}_1-v_i$ together with edge $e$ and $f$, and the component $\widetilde{H}[V_i]$ of $H$ corresponding to $V_i$ is a unicyclic spanning subgraph of $\widetilde{G}_i$ for $i=2, \ldots, p$, which is preserved while we add $e$ to $G$.

Therefore, we can deduce that
\begin{equation} \label{add1}
\frac{\sum_{\widetilde{H} \in \sigma(\mathscr{H}_{\pi}(G))} \omega(\widetilde{H})}{\sum_{H \in \mathscr{H}_{\pi}(G) } \omega(H)}
=\frac{2^p |\mathscr{T}(G_1)| \Pi_{i=2}^p|\mathscr{H}_1(\widetilde{G}_i)|}{2^p\Pi_{i=1}^p|\mathscr{H}_1(G_i)|}=\frac{|\mathscr{T}(G_1)|}{|\mathscr{H}_1(G_1)|}\geq \frac{g(G_1)}{c(G_1)}>\frac{7}{6},
\end{equation}
where the last inequality follows from Lemma \ref{cou} and Remark \ref{rmk}.

If $\pi$ is of Type II, assume that $v_i \in V_1$ and $v_j \in V_2$. We still use $G_i$ ($\widetilde{G}_i$, resp.) to denote  the subgraph of $G$ ($G+e$, resp.) induced by $V_i$ for $i =1, \ldots, p$.
In the situation,  $\sigma(\mathscr{H}_{\pi}(G))$ is defined to be the subset of $\mathscr{H}^{e}(G+e)$ consisting of the TU-subgraphs $\widetilde{H}$ of size $n$ satisfying the following conditions:
\begin{itemize}
	\item The components of $\widetilde{H}$  are corresponding to the partition $\{V_1\cup V_2, V_3, \ldots, V_p\}$;
	\item $\widetilde{H}[V_1 \cup V_2]-e$ exactly have two components which correspond to $V_1$ and $V_2$ respectively. Equivalently, $e$ connects a spanning tree of $\widetilde{G}_1$ and a unicyclic spanning subgraph of $\widetilde{G}_2$, or $e$ connects a spanning tree of $\widetilde{G}_2$ and a unicyclic spanning subgraph of $\widetilde{G}_1$.
\end{itemize}
Note that $G_i=\widetilde{G}_i$ for $i=1, \ldots, p$ as $\pi$ is of Type II.
Therefore, we  can deduce that
\begin{equation} \label{add5}
\begin{split}
\frac{\sum_{\widetilde{H} \in \sigma(\mathscr{H}_{\pi}(G))} \omega(\widetilde{H})}{\sum_{H \in \mathscr{H}_{\pi}(G) } \omega(H)}&=\frac{2^{p-1}(|\mathscr{T}(\widetilde{G}_1)||\mathscr{H}_1(\widetilde{G}_2)|+|\mathscr{T}(\widetilde{G}_2)||\mathscr{H}_1(\widetilde{G}_1)|) \Pi_{i=3}^p|\mathscr{H}_1(\widetilde{G}_i)|}{2^p\Pi_{i=1}^p|\mathscr{H}_1(G_i)|}\\
&=\frac{|\mathscr{T}(G_1)|}{2|\mathscr{H}_1(G_1)|}+\frac{|\mathscr{T}(G_2)|}{2|\mathscr{H}_1(G_2)|} >\frac{7}{6},
\end{split}
\end{equation}
where the last inequality follows from Lemma \ref{cou} and Remark \ref{rmk}.

We are now ready to prove \eqref{add2}. By  the above definition of   $\sigma(\mathscr{H}_{\pi}(G))$, it is clear that $\sigma(\mathscr{H}_{\pi}(G)) \cap\sigma(\mathscr{H}_{\pi'}(G)) =\emptyset$ if $\pi \neq \pi'$.
Note that  $\omega(\mathscr{H}(G))= \sum_{\pi \in \Omega}\sum_{H \in \mathscr{H}_{\pi}(G) } \omega(H)$ and $\omega(\mathscr{H}^{e}(G+e))= \sum_{\pi \in \Omega}\sum_{\widetilde{H} \in \sigma(\mathscr{H}_{\pi}(G)) }  \omega(\widetilde{H})$. To establish \eqref{H(G+e)}, it suffices to show that $$\frac{\sum_{\widetilde{H} \in \sigma(\mathscr{H}_{\pi}(G))} \omega(\widetilde{H})}{\sum_{H \in \mathscr{H}_{\pi}(G) } \omega(H)}> \frac{7}{6}$$  for all TU-admissible $\pi \in \Omega$, which has been provided by \eqref{add1} and \eqref{add5}. The result follows.
\end{proof}

The following result is a consequence of Theorem \ref{mainthm}.
\begin{corollary}
Let $G$ be a connected unicyclic or bicyclic graph. Then, the LMRIV will not occur to $G$ in two places by adding a new edge.
\end{corollary}
We finish the paper with the conjecture below which has been verified by computer for $n \leq 9$.
\begin{conjecture}
	Let $G$ be a connected graph of order $n$. The LMRIV will not occur to $G$ in two places by adding a new edge.
\end{conjecture}

\section*{Acknowledgments} We are grateful to the referees for their comments and suggestions.


\end{document}